\documentclass[letterpaper, 12pt]{article}
\usepackage{amsmath,amsthm,amssymb,enumerate,mathscinet}
\usepackage[bookmarks]{hyperref}
\usepackage{tikz}
\usepackage{fullpage}
\usepackage{setspace}
\usepackage{listings}
\usepackage{verbatim}

\newcount\sracnum
\def\comment#1{{
\global\advance \sracnum by 1 ${}^{[\the\sracnum]}$
\marginpar{%
\vskip-4mm{\tiny~~(\the\sracnum)}~~
 {\footnotesize \baselineskip=10pt \raggedright #1\\ ~\\~}}}}

\newtheorem{thm}{Theorem}

\newtheorem{conjecture}[thm]{Conjecture}

\newtheorem{prop}[thm]{Proposition}

\newtheorem{lem}[thm]{Lemma}

\theoremstyle{definition}
\newtheorem{claim}[thm]{Claim}

\newtheorem*{defin*}{Definition}

\def\text#1{\quad\textnormal{#1}\quad}

\def\Text#1{\qquad\textnormal{#1}\qquad}

\def\e{\mathbb{E}}

\newcommand{\OO}[1]{\overline{#1}}

\newcommand{\floor}[1]{\left\lfloor #1 \right\rfloor}

\def\M{\mathcal{M}}
\def\L{\mathcal{L}}
\def\N{\mathbb{N}}

\def\s{\sigma}

\def\HH{\mathcal{H}}

\def\F{\mathcal{F}}

\def\FS{\F^{\sigma}}

\def\P{P}

\def\eps{\varepsilon}

\def\im{\textnormal{Im}}

\def\Remark{{\noindent \textbf{Remark:}} }

\def\beq{\begin{equation}}

\def\eeq{\end{equation}}

\def\bth{\begin{thm}}

\def\eth{\end{thm}}

\def\blm{\begin{lem}}

\def\elm{\end{lem}}

\def\bc{\begin{corollary}}

\def\ec{\end{corollary}}

\def\bcj{\begin{conjecture}}

\def\ecj{\end{conjecture}}

\def\bpf{\begin{proof}}
\def\epf{\end{proof}}

\def\bpp{\begin{prop}}
\def\epp{\end{prop}}

\def\pg{representation graph}

\newcommand{\oururl}{\url{http://www.math.uiuc.edu/~jobal/cikk/permutations/}}

\DeclareMathOperator{\ourmod}{mod}

\title{Minimum number of  monotone subsequences of length 4 in permutations}
\author{
 J\'{o}zsef Balogh\thanks{  Department of Mathematics, University of Illinois, Urbana, IL 61801, USA and Bolyai Institute, University of Szeged, Szeged, Hungary {\tt jobal@math.uiuc.edu}.
    Research is partially supported by Simons Fellowship, NSF CAREER Grant DMS-0745185, Arnold O. Beckman Research Award (UIUC Campus Research Board 13039) and Marie Curie FP7-PEOPLE-2012-IIF 327763.}
 \and Ping Hu\thanks{Department of Mathematics, University of Illinois, Urbana, IL 61801, USA, {\tt pinghu1@math.uiuc.edu}.}
 \and Bernard Lidick\'{y}\thanks{Department of Mathematics, University of Illinois, Urbana, IL 61801, USA, {\tt lidicky@illinois.edu} and Charles University in Prague. Research is partially supported by NSF grant DMS-1266016.}
 \and Oleg Pikhurko\thanks{Mathematics Institute and DIMAP, University of Warwick, Coventry CV4 7AL, UK, {\tt O.Pikhurko@warwick.ac.uk}. Research is partially supported by ERC grant 306493 and EPSRC grant EP/K012045/1.}
 \and Bal\'{a}zs Udvari\thanks{
Bolyai Institute, University of Szeged, Szeged, Hungary. {\tt udvarib@math.u-szeged.hu}.
Research was supported by the European Union and the State of Hungary, 
co-financed by the European Social Fund in the framework of T\'{A}MOP 4.2.4. 
A/2-11-1-2012-0001 National Excellence Program.
 }
 \and Jan Volec\thanks{
 Mathematics Institute and DIMAP, University of Warwick, Coventry CV4 7AL, UK, {\tt honza@ucw.cz}.
  Research is partially supported by the European Research Council under the European
  Union's Seventh Framework Programme (FP7/2007-2013)/ERC grant agreement no.~259385.
 }
}

\begin{document}
\maketitle

\begin{abstract}
{
We show that for every sufficiently large $n$, the number of monotone subsequences
of length four in a permutation on $n$ points is at least
$\binom{\floor{n/3}}{4} + \binom{\floor{(n+1)/3}}{4} + \binom{\floor{(n+2)/3}}{4}$.
Furthermore, we characterize all permutations on $[n]$ that attain this lower bound.
The proof uses the flag algebra framework together with some additional stability arguments. 
This problem is equivalent to some specific type of edge colorings of complete graphs with two
colors, where the number of monochromatic $K_4$'s is minimized.
We show that all the extremal colorings must contain monochromatic $K_4$'s only in one of the two colors.
This translates back to permutations, where all the monotone subsequences of length four are all
either increasing, or decreasing only.
}
 \end{abstract}

\section{Introduction}
Our work was inspired by a famous result of Erd\H{o}s and Szekeres~\cite{ErdosS35}
that every permutation on $[n] = \{1,\dots,n\}$, where $n \geq k^2+1$, contains a
monotone subsequence of length $k+1$.
If $n \gg k^2$, one expects that the number of monotone subsequences of length $k+1$
is more than just one, which is guaranteed by~\cite{ErdosS35}. 
According to Myers~\cite{Myers02}, the problem of determining the minimum number
of monotone subsequences of length $k+1$ in permutations on $[n]$ was first posed
by Atkinson, Albert and Holton. 
As in~\cite{Myers02}, we use $m_k(\tau)$ to denote the number of monotone subsequences of length $k+1$
in a permutation $\tau$. The minimum of $m_k(\tau)$ over all permutations $\tau \in S_n$
is denoted by $m_k(n)$.

Myers~\cite{Myers02} described a permutation $\tau_k(n)$ which gives an upper
bound on $m_k(n)$. It consists of $k$ increasing sequences $K$ whose
sizes differ by at most one and every monotone sequence
of length $k+1$ is entirely contained in one of the $K$ sequences. 
In other words, with $t_j = \lfloor jn/k \rfloor$, an example of such a permutation is 
\[
\begin{array}{clllllll}
 \tau_k(n)=(&t_{k-1} +1,& t_{k-1}+2,&\ldots,&n-1,& n,&\\
                  & t_{k-2}+1,& t_{k-2}+2,&\ldots,&t_{k-1}-1,&t_{k-1},&\\
                  &\ldots        &                 &         &               &\\
                  &1,              &2,              &\ldots,&t_{1}-1,& t_1&).
\end{array}
\]
See Figure~\ref{fig-Tn} for $\tau_3(12)$.
\begin{figure}[htpd]
\begin{center}
\includegraphics{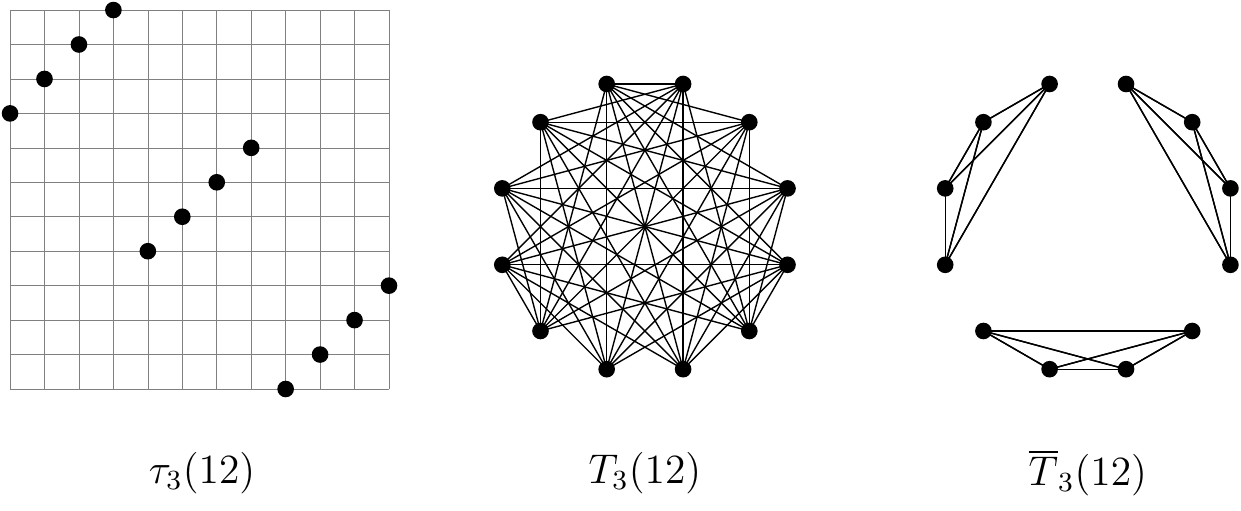}
\end{center}
\caption{Permutation $\tau_3(12)$ and its {\pg} (introduced in Section~\ref{sec:GraphDensity}) ${T}_3(12)$.}\label{fig-Tn}
\end{figure}

Let $r \equiv n \pmod k$, where $0 \leq r < k$. It is easy to see that
\begin{align*}
m_k(\tau_k(n)) = r \binom{\lceil n/k \rceil}{k+1} + (k - r)\binom{\lfloor n/k \rfloor}{k+1} \approx \frac{1}{k^k}\binom{n}{k+1}. \label{eq:mk}
\end{align*}

Myers~\cite{Myers02} proved that $m_2(n) = m_2(\tau_2(n))$ holds and he described all permutations $\tau \in S_n$ where
$m_2(\tau) = m_2(n)$.
{ He conjectured that the same formula actually holds for every $k\in\N$.}
\begin{conjecture}[Myers~\cite{Myers02}]\label{con:exact}
  Let $n$ and $k$ be positive integers. 
  In any permutation of $[n]$ there are at least $m_k(\tau_k(n))$ monotone subsequences of length $k+1$.
\end{conjecture}
Notice that any permutation $(a_1,\ldots,a_n)$ and its reverse $(a_n,\ldots,a_1)$ contain the same number
of monotone subsequences, only the increasing subsequences change to decreasing subsequences and vice versa. 
In particular, $m_k(\tau_k(n)) = m_k(\tau_k^R(n))$, where $\tau_k^R(n)$ denotes the reverse of $\tau_k(n)$.
Moreover, there might be other permutations $\tau$ such that $m_k(\tau)=m_k(\tau_k(n))$.

As we already mentioned, Myers showed the conjecture is true for $k=2$,
which is actually a consequence of Goodman's formula~\cite{Goodman:1959}.
Very recently, Samotij and Sudakov~\cite{SamoSud:2014} confirmed the conjecture if
$n \le k^2 + ck^{3/2}/\log k$ for some absolute positive constant $c$, provided
$k$ is sufficiently large.

Subject to the additional constraint that all the monotone subsequences of
length $k+1$ are either all increasing or all decreasing and $n \geq k(2k-1)$,
Myers proved that every such a permutation contains at least the conjectured number 
of monotone subsequences of length $k+1$.
He also gave the list $\mathcal{W}_n^k$ of all such permutations $\tau$ of
$[n]$ that satisfy $m_k(\tau) = m_k(\tau_k(n))$.  Every permutation from
$\mathcal{W}_n^k$ can be decomposed into $k$
disjoint monotone subsequences $s_1,\ldots,s_k$ that are either all increasing or all decreasing and
their sizes differ by at most one.
Moreover, every monotone subsequence of length $k+1$ is a subsequence of $s_j$ for some $j$.
These permutations look similar to  $\tau_k(n)$ or $\tau_k^R(n)$.
It turns out that there are $2\binom{k}{n \ourmod k}C_k^{2k-2}$ of them, where $C_k$ is the $k^{\rm th}$ Catalan number. 

The interested reader can find the precise definition of $\mathcal{W}_n^k$ for general $k$ in~\cite{Myers02}. 
Here, we study the number of monotone subsequences with  $k=3$.
Hence we give a simpler alternative definition for $\mathcal{W}_n^3$, where $n \geq 15$.

First we describe a method to get any permutation from $\mathcal{W}_n^3$
with no increasing subsequence of length $4$.

\begin{enumerate}
\item Start with the identity permutation.
\item Divide it into $3$ blocks such that 
	the size of each block is $\lfloor n/3 \rfloor$ or $\lfloor n/3\rfloor+1$. 
	More formally, choose elements $b_1$ and $b_2$ such that $b_1$, $b_2-b_1$ and $n-b_2$ are all from the set $\{\lfloor n/3\rfloor,\lfloor n/3\rfloor+1\}$.
	Then the three blocks are $(1,2,\ldots,b_1)$, $(b_1+1,b_1+2,\ldots,b_2)$, $(b_2+1,b_2+2,\ldots,n)$. 
	(There are $1$ or $3$ choices for the pair $(b_1,b_2)$, depending on the remainder of dividing $n$ by $3$.)
\item Reverse the blocks. At this point we have the permutation 
	$(b_1,b_1-1,\ldots,2,1,b_2,b_2-1,\ldots,b_1+2,b_1+1,n,n-1,\ldots,b_2+2,b_2+1)$.
\item Change the subsequence $(2,1,b_2,b_2-1)$ to one of the following:
	$(2,1,b_2,b_2-1)$, $(2,b_2,1,b_2-1)$, $(2,b_2,b_2-1,1)$, 
	$(b_2,2,1,b_2-1)$, $(b_2,2,b_2-1,1)$.
\item Make a similar replacement for the subsequence $(b_1+2,b_1+1,n,n-1)$.
\item Change the subsequence $(b_1,b_1-1,b_1+2,b_1+1)$ to one of the following:
	$(b_1,b_1-1,b_1+2,b_1+1)$, $(b_1+1,b_1-1,b_1+2,b_1)$, $(b_1+2,b_1-1,b_1+1,b_1)$,  
	$(b_1+1,b_1,b_1+2,b_1-1)$, $(b_1+2,b_1,b_1+1,b_1-1)$.
\item Make a similar replacement for the subsequence $(b_2,b_2-1,b_2+2,b_2+1)$.
\end{enumerate}

Each above permutation (as well as its reverse) belongs to $\mathcal{W}_n^3$ since it has 
$m_3(\tau_3(n))$ monotone subsequences of length $4$, all of which are decreasing. For $n\ge 15$, 
we exhaust all of $\mathcal{W}_n^3$, as the number of obtained permutations, $2 \cdot 5^4 
\binom{3}{r}$ where $r$ is the remainder of dividing $n$ by $3$, coincides with the value of 
$|\mathcal{W}_n^3|$ obtained by Myers.

To illustrate the above process, let $n = 17$. 
We start with $(1, 2, \ldots, 17)$. Let $b_1 = 5, b_2 = 11$.
After the reversal of the blocks, we have $(5,4,3,2,1,11,10,9,8,7,6,17,16,15,14,13,12)$. 
Now we can change, one by one in the given order, the subsequences 
$(2,1,11,10)$, $(7,6,17,16)$, $(5,4,7,6)$, $(11,10,13,12)$ to
$(11,2,1,10)$, $(17,7,6,16)$, $(6,5,7,4)$, $(13,10,12,11)$ respectively,
to get 
\[(6,5,3,13,2,1,10,9,8,17,7,4,16,15,14,12,11).\]
\begin{figure}
\begin{center}
\includegraphics{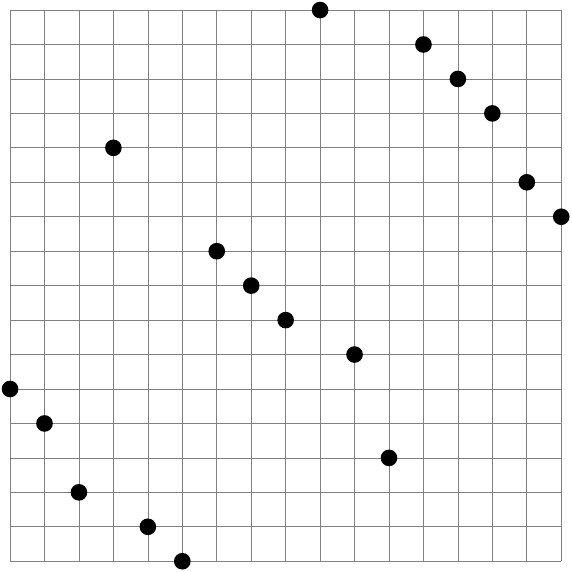}
\end{center}
\caption{Permutation (6,5,3,13,2,1,10,9,8,17,7,4,16,15,14,12,11).}\label{fig-extremal}
\end{figure}
This permutation is depicted in Figure~\ref{fig-extremal}.

In his paper, Myers~\cite{Myers02} also conjectured a weaker asymptotic version.
\begin{conjecture}[Myers~\cite{Myers02}]\label{con:asym}
  Let $k$ be positive integer and let $n \rightarrow \infty$.
  In any permutation of $[n]$ there are at least $(1+o(1))\binom{n}{k+1}/k^k$ monotone subsequences of length $k+1$.
\end{conjecture}

First, we prove Conjecture~\ref{con:asym} for $k=3$.
\bth\label{thm:asym} 
Any permutation of $[n]$ contains at least $(1/27+o(1))\binom{n}{4}$ monotone subsequences of length $4$.
\eth

Our  main result is proving Conjecture~\ref{con:exact} for $k=3$ and $n$ sufficiently large.

\bth\label{thm:exact}
There exists $n_0$ such that if $n \geq n_0$, then every permutation $\tau$ on $[n]$ contains at least $$\binom{\floor{n/3}}{4} + \binom{\floor{(n+1)/3}}{4} + \binom{\floor{(n+2)/3}}{4}$$ monotone subsequences of length $4$, with equality if and only if $\tau \in \mathcal{W}_n^3$.
\eth

Our results are proved using the flag algebra {framework} and the stability method.
Although Theorems~\ref{thm:asym} and \ref{thm:exact} are stated in terms of permutations, we translate
them to the language of graph theory since the resulting computations and arguments are simpler.
In graph theory language, we minimize the number of copies of $K_4$ and  $\OO K_4$ over graphs from permutations on $[n]$.
Let us note that the question of  minimizing the number of copies of $K_4$ and  $\OO K_4$ over all graphs on $n$ vertices is open.
The best upper bound $\approx 1/33$ is due to Thomason~\cite{Thomason:1989}.
The first known lower bound $\approx 1/46$ is due to Giraud~\cite{Giraud}.
It was improved using flag algebras to $0.0287...$ by Sperfeld~\cite{SperfeldK4:2011} and independently by Nie{\ss}~\cite{Niess:2012},
and then further improved by Flagmatic~\cite{Flagmatic} to $0.0294... \approx 1/34$. 

{We also had a computer program, developed originally by Dan Kr\'a\soft{l}, doing flag algebra
computations for permutations directly.}
It was easy to modify this program to compute upper bounds 
on densities of other subsequences instead of lower bounds for monotone subsequences. The results that we obtained will be explained in the next paragraph.

The \emph{packing density} of a permutation $\tau \in S_k$ is the limit for $n\rightarrow \infty$
of the maximum density of $\tau$ in $\sigma$ over all $\sigma \in S_n$. We denote the limit by $\delta(\tau)$.
The packing density is well understood~\cite{AlbertAHHS02} for the so-called \emph{layered permutations}\footnote{
A permutation $\tau \in [n]$ is \emph{layered} if there exist positive numbers $n_1,\ldots,n_r$ summing to $n$,
such that $\tau$ starts with the $n_1$ first positive integers in reverse order, followed by the next $n_2$ positive integers in reverse order and so on. For example $\tau_k^R(n)$ is a layered permutation.}.
Up to a symmetry, this includes all permutations in $S_3$ and all but two permutations, 1342 and 2413, from $S_4$.
Albert, Atkinson, Handley, Holton, and Stromquist~\cite{AlbertAHHS02} proved that $0.19657 \leq \delta(1342) \leq 2/9$
and $51/511 \leq \delta(2413) \leq 2/9$.
Presutti~\cite{Presutti08} improved the lower bound  for $\delta(2413)$ to 0.1024732.
Further improvement on the lower bound was obtained by Presutti and Stromquist~\cite{Presutti10}
who showed that $0.1047242275767320904\ldots \leq \delta(2413)$ and conjectured that it is the correct value.
A direct application of {the semidefinite method from the flag algebra framework} for permutations on $S_7$ gave upper 
bounds $\delta(1342) \leq 0.1988373$ and $\delta(2413) \leq 0.1047805$.
Since our upper bounds do not match the lower bounds, we will not discuss
these bounds any further in this paper.

This paper is organized as follows. In the following section,
 we translate the problem of determining the density of monotone subsequences in permutations to determining densities of particular
induced subgraphs in permutation graphs.
In Section~\ref{sec:FA}, we describe how we use the framework of flag algebras and we will prove Theorem~\ref{thm:asym}.
Our proof of the density result actually provides some additional information about the extremal structures,
which leads to a proof of a stability property for this problem. This is discussed in Section~\ref{sec:stability}.
Finally, in Section~\ref{sec:exact}, we use the stability property to prove Theorem~\ref{thm:exact}.

We utilize the semidefinite method from flag algebras to formulate our question about subgraph densities
as an optimization problem, more precisely, as a semidefinite programming problem. With a computer assistance, we generate
this semidefinite programming problem and then we use CSDP~\cite{Borchers:1999}, an open-source semidefinite
programming library, to find a numerical (approximate) solution to the problem. In order to obtain an exact result,
the numerical solution needs to be rounded. This was done again with a computer assistance
in a computer algebra software SAGE~\cite{sw:sage}.
We had trouble finding a detailed description of rounding in other papers.
Hence we decided to include more details about our rounding procedure in the
appendix.

Our computer programs, their outputs, and their description for the flag algebra part of this paper can be downloaded at
\oururl.

\section{Graph Densities}\label{sec:GraphDensity}
Given a graph $G$, we use $V(G)$ and $E(G)$ to denote its vertex and edge sets respectively, and let $v(G)=|V(G)|, e(G) = |E(G)|$. For a vertex $v$ of $G$, we denote the set of its neighbors by $\Gamma_G(v)$.
We omit a subscript, if $G$ is clear from the context. Given two graphs $G$ and $G'$, an \emph{isomorphism} between them is a  bijection $f: V(G) \to V(G')$ satisfying $f(v_1)f(v_2)\in E(G')$ if and only if $v_1v_2\in E(G)$.  
Two graphs $G$ and $G'$ are \emph{isomorphic} $(G \cong G')$ if and only if there is an isomorphism between them.
For a graph $G$ and a vertex set $U\subseteq V(G)$, denote by $G[U]$ the induced subgraph of $G$ on vertex set $U$.
Suppose $H$ and $G$ are graphs on $l$ and $n$ vertices respectively. Let $\P(H, G)$ be the number of $l$-subsets $U$ of $V(G)$ such that $G[U]\cong H$, and define the \emph{density} of $H$ in $G$ to be
\[
p(H, G) = \frac{\P(H, G)}{\binom{n}{l}}.
\]

Given a permutation $\tau$ of $[n]$, define its \emph{\pg}
to be a graph on vertex set $[n]$ where $ij$ with $i<j$ is an edge if and only if $\tau(i)>\tau(j)$. Call an $n$-vertex graph $G$
\emph{admissible} if there is a permutation of $[n]$ whose {\pg} is isomorphic to $G$, so the vertex set of $G$ may not be $[n]$.
Denote by 
$\M_l$ the set of admissible graphs on $l$ vertices, up to isomorphism.
It is easy to see that if $G$ is admissible, then so are $\OO G$ and all induced subgraphs of $G$.  

Given a permutation $\tau$ of $[n]$, let $G$ be its {\pg}. Then the number of monotone subsequences of length $4$ in $\tau$ is equal to the number of $K_4$'s and $\OO K_4$'s in $G$, i.e., $m_3(\tau)=\P(K_4, G) + \P({\OO K_4}, G)$.
Let 
\[ 
F(G) = \P(K_4, G) + \P({\OO K_4}, G)\Text{and} f(G) = p(K_4, G) + p({\OO K_4}, G).
\]
Instead of proving Theorem~\ref{thm:asym} directly, we prove its reformulation to the language of graphs and densities.
\bth\label{thm:asymG}
If $G$ is an admissible graph on $n$ vertices, then $f(G)\ge 1/27+o(1)$, where $o(1) \to 0$ as $n\to \infty$.
\eth

It is easy to see that

\beq\label{eq:trivilBound}
f(G) = \sum_{H\in \M_l}f(H)p(H,G)\text{for}4\le l\le n.
\eeq

Therefore $\min_{H\in\M_l}f(H)$ provides a lower bound on $f(G)$ (since $0\le p(H,G)\le 1$ and $\sum_{H\in \M_l}p(H,G) = 1$), though this bound is unsurprisingly weak for small $l$.

Denote by $T_3(n)$ the  $3$-partite Tur\'an graph on $n$ vertices (i.e.  complete $3$-partite graph on $n$ vertices with sizes of parts
differing by at most one). 
We can see that $T_3(n)$ is the {\pg}  of $\tau_3(n)$. See Figure~\ref{fig-Tn} for an example, where $n=12$.
\bth\label{thm:exactG}
There exists an $n_0$ such that 
if $G$ is an admissible graph on $n \geq n_0$ vertices 
minimizing $F$ over all admissible graphs on $n$ vertices,
then $G$ is obtained from $T_3(n)$ by removing edges
or $G$ is obtained from $\OO{T_3(n)}$ by adding edges.
\eth
\Remark 
Let $G$ be an extremal graph. 
By Theorem~\ref{thm:exactG}, $G$ can be transformed into $T_3(n)$ or $\OO{T_3(n)}$. 
We may assume without loss of generality (w.l.o.g.) that $G$ is obtained from $T_3(n)$ by removing edges.
Since $T_3(n)$ does not contain any copy of $K_4$ and removing edges does not introduce new copies of $K_4$, there are no $K_4$'s in $G$.
Moreover, since $G$ is extremal and removing edges does not destroy any copy of $\OO K_4$, the numbers of copies of $\OO K_4$ in $G$ and $T_3(n)$ are equal.
Hence we know that in an extremal permutation $\tau$, monotone subsequences of length $4$ are either all increasing or all decreasing. Thus $\tau$ belongs to the family $\mathcal{W}_n^3$
constructed by Myers (and Theorem~\ref{thm:exact} follows from Theorem~\ref{thm:exactG}). 
In fact, it is not hard to see that $\tau\in \mathcal{W}_n^3$ directly. Indeed,
$\tau$ can be decomposed into three monotone subsequences $s_1,s_2,s_3$, that correspond to the parts of Tur\'an graph,
and all monotone 4-subsequences are entirely contained in them. Then it follows that the domains of $s_1,s_2,s_3$ form three consecutive intervals of $[n]$, except some possible intertwining at their
ends that involves at most two elements from each interval, which leads to the desired structure of $\tau$.

\section{Flag Algebra Settings}\label{sec:FA}
The flag algebra method, invented by Razborov~\cite{Razborov:2007},   is a very general machinery and has been widely used in extremal graph theory. 
See~\cite{Razborov13} for a recent survey of flag algebra applications.
To name just some of them: flag algebra was used for attacking the
Caccetta-H\"aggkvist conjecture~\cite{HladkyKN:2009,RazborovCH:2011},
determining induced densities of graphs~\cite{DasHMNS:2012,Grzesik:2011,Hatami:2011,PikhurkoR:2012,PikhurkoV:2013},
of hypergraphs~\cite{BaberT:2011,Falgas:2011,GlebovKV:2013,Pikhurko:2011},
of oriented graphs~\cite{Sperfeld:2011},
of subhypercubes in hypercubes~\cite{Baber:2012,BaloghHLL:2014},
of colored graphs in a colored environment~\cite{BaberT:2013,CummingsKPSTY:2012,HatamiJKNR:2012,KralLSWY:2012}, and 
for attacking some problems in geometry~\cite{Kral:2011}.

We apply this method to the family of admissible graphs.
A \emph{type} $\s$ is an admissible graph on vertex set $[k]$
for some non-negative integer $k$, where $k$ is called the \emph{size} of $\s$, denoted by $|\s|$. We use $0$ and $1$ to denote (the unique) types of size $0$ and $1$ respectively.
A \emph{$\sigma$-flag} $F$ is a pair $(M,\theta)$ where $M$ is an admissible graph and $\theta:[k]\to V(M)$ induces a labeled copy of $\s$ in $M$. In other words, we use $[k]$ to label $k$ vertices of an unlabeled graph $M$, and the labeled vertices induce a labeled copy of $\s$.
Two $\s$-flags $F_1 = (M_1, \theta_1)$ and $F_2=(M_2, \theta_2)$ are \emph{isomorphic} (denoted as $F_1\cong F_2$) if there exists a graph isomorphism $f:V(M_1)\to V(M_2)$ such that $f\theta_1 = \theta_2$. Such a function $f$ is called a \emph{flag isomorphism} from $F_1$ to $F_2$. Given an admissible graph $M$, if all $\s$-flags with the underlying graph $M$ are isomorphic, then we use $M^{\sigma}$ to denote this unique $\s$-flag, see Figure~\ref{fig-K4} for an example where $M\in \{K_4,\overline{K}_4\}$.
Denote by $\FS_l$ the set of $\s$-flags on $l$ vertices, up to isomorphism. Note that $\F_l^0$ is just $\M_l$ and $\FS_{|\s|} = \{\s\}$.

\begin{figure}[htpd]
\begin{center}
\includegraphics{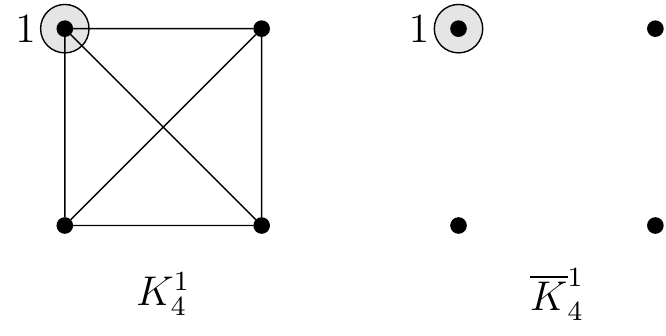}
\end{center}
\caption{$1$-flags $K_4^1$ and $\OO K_4^1$.}\label{fig-K4}
\end{figure}

In Section~\ref{sec:GraphDensity}, we defined graph density $p(H,G)$, which extends to flag density in a straightforward way. Given $\s$-flags $F\in\FS_l$ and $K = (G,\theta)\in \FS_n$ for $l\le n$, define $P(F,K)$ to be the number of $l$-subsets $U$ of $V(G)$ such that $\im(\theta)\subseteq U$ and $(G[U],\theta)\cong F$. Additionally, define $p(F,K)$, the \emph{density} of $F$ in $K$ as
\[
p(F, K) = \frac{P(F, K)}{\binom{n-|\s|}{l-|\s|}}.
\]
By convention, we set $P(F,K) = 0$ if $n<l$.
More generally, given flags $F\in \FS_{l}$, $F'\in \FS_{l'}$ and $K = (G,\theta)\in \FS_n$, where $n \ge l+l'-|\sigma|$, we define a \emph{joint density} $p(F,F'; K)$ as the probability that if we choose two subsets $U, U'$ of $V(G)$ uniformly at random, subject to the conditions $|U|=l, |U'|=l'$ and $U\cap U' = \im(\theta)$, then $(G[U],\theta)\cong F$ and $(G[U'],\theta)\cong F'$. 
In this paper, whenever we use $p(F,K)$ or $p(F,F';K)$, we assume that the size of $K$ is large enough.

It is not very hard to show that (see Lemma 2.3 in \cite{Razborov:2007})
\beq\label{eq:product}
p(F,K)p(F',K) = p(F,F';K) + o(1),
\eeq
where $o(1)$ tends to $0$ as $n$ tends to infinity.
Let $X = [F_1,\ldots, F_t]$ be a vector of $\s$-flags with $F_i\in\FS_{l_i}$.
For any such $X$ and a $\s$-flag $K$ define $X_K = [p(F_1;K),\ldots, p(F_t;K)]$.
It follows that for any $t$-by-$t$ positive semidefinite matrix $Q = \{Q_{ij}\}$, we have
\beq\label{eq:Q}
0 \leq X_K^TQX_K = \sum_{ij}Q_{ij}p(F_i;K)p(F_j;K) = \sum_{i,j}Q_{ij}p(F_i,F_j;K) +o(1).
\eeq
In the definition of $p(F,K)$ and $p(F,F';K)$, we require $F,F'$ and $K$ to be $\s$-flags, but the definition itself extends to the case where $F,F'$ are $\s$-flags but $K$ is not. In this case, by the definition, we have $p(F,K) = p(F,F'; K) = 0$. Let $\Theta(k, G)$ be the set of all injective mappings from $[k]$ to $V(G)$ where $G$ is an admissible graph. 

We can extend \eqref{eq:Q} to any $\theta\in\Theta(|\s|, G)$:
\[
0\le\sum_{i,j}Q_{ij}p(F_i,F_j;(G,\theta)) +o(1).
\]
Therefore, if we choose $\theta$ from $\Theta(|\s|, G)$ uniformly at random, then its expectation is non-negative:
\begin{align*}
0&\le \sum_{i,j}\e_{\theta\in\Theta(|\s|, G)}\left[Q_{ij}p(F_i,F_j;(G,\theta))\right] +o(1) \\
&= \sum_{H\in \M_l}\left( \sum_{i,j}\e_{\theta\in\Theta(|\s|, H)}\left[Q_{ij}p(F_i,F_j;(H,\theta))\right]\right)p(H,G) +o(1).
\end{align*}
(Recall that we assumed that $l\geq 2l_i-|\s|$ for each $i$.)
Note that the coefficient of $p(H,G)$ is determined by $\s,X,Q$ and $H$. In particular, it is independent of $G$, so denote this coefficient by $c_H(\s,X,Q)$. Then we have 
\[
\sum_{H\in \M_l}c_H(\s,X,Q)p(H,G)+o(1) \ge 0.
\]

Every choice of $\s, X, Q$ gives one such inequality. We can add the inequalities obtained
for several different types $\s_i$, using appropriate $X_i$ and $Q_i$.
Denoting $c_H = \sum_i c_H(\s_i,X_i,Q_i)$, we obtain
\[
\sum_{H\in \M_l}c_H \cdot p(H,G)+o(1) \ge 0.
\]
Then together with \eqref{eq:trivilBound} we have 
\beq\label{eq:lowerBound}
f(G) +o(1) \ge \sum_{H\in \M_l} (f(H)-c_H)\cdot p(H,G) \geq  \min_{H\in\M_l} (f(H)-c_H).
\eeq
By~\eqref{eq:lowerBound}, if for some choice of (large enough) $l$ and $c_H$ we have 
\begin{align}
\min_{H\in\M_l} (f(H)-c_H) = 1/27, \label{eq:27}
\end{align}
then we would prove Theorem~\ref{thm:asymG}.

\begin{proof}[Proof of Theorem~\ref{thm:asymG}]
We show \eqref{eq:27} with $l=7$, where $|\M_7| = 776$. We use three choices of $(\sigma, X,Q)$.
We use types $\s_0: P_1$, $\s_1: P_3$, and $\s_2: {\OO P_3}$, where $P_i$ is a path on $i$ vertices, see Figure~\ref{fig-types}.

\begin{figure}[htpd]
\begin{center}
\includegraphics{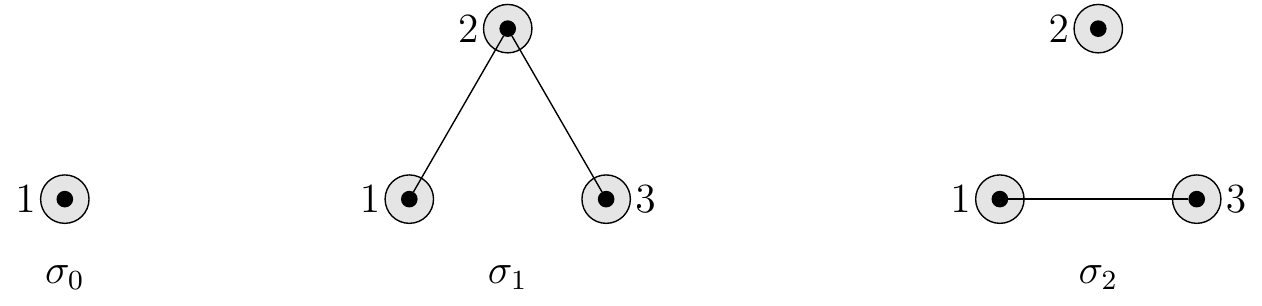}
\end{center}
\caption{Types used in flag computation.}\label{fig-types}
\end{figure}

For $\s_0$, $X_0$ consists of flags in $\F^{\s_0}_4$, for $\s_i$ with $i=1,2$, $X_i$ consists of flags in $\F^{\s_i}_5$.
Here we have $|\F^{\s_0}_4| = 20$ and $|\F^{\s_1}_5| = |\F^{\s_2}_5| = 71$.
As we already mentioned, the flag algebra method is computer assisted. We use a computer program to find $\M_7, \F^{\s_0}_4, \F^{\s_1}_5, \F^{\s_2}_5$, and to compute $\e_{\theta}~ p(F,F';(H,\theta))$ for each $H\in\M_y$. Then finding positive semidefinite matrices $Q_0, Q_1, Q_2$ to maximize $\min_{H\in \M_7} (f(H)-c_H)$ can be done by computer solvers such as CSDP~\cite{Borchers:1999} and SDPA~\cite{Yamashita10ahigh-performance}.
Unfortunately, solvers can only give an approximate solution.
For this problem, we get $0.0370370369999$. 
In order to get exactly $1/27$, we need to round the matrices $Q_0,Q_1,Q_2$ found by a computer solver. 
By rounding we mean finding  rational matrices $Q_0',Q_1',Q_2'$ which would make the computations exactly $1/27$ when computed over rational numbers.

To simplify the process of rounding, we reduce the number of variables and constraints by restricting the set of feasible solutions. 
For $i\in \{1,2,3\}$ and flags $F_1,F_2$ denote by  $Q_i(F_1,F_2)$ the entry in $Q_i$ corresponding
to indices of $F_1$ and $F_2$ in $X_i$.
Since $f(H) = f(\overline H)$ for every graph $H$, a natural restriction
is that
\begin{align}
 f(H) - c_H = f(\overline H) - c_{\overline H}\label{eq:equal}
\end{align}
for every graph $H$. This will allow us to consider only one of $H,\overline{H}$ and thus decrease the number
of constraints from 776 to 388 since there is no self-complementary graph on 7 vertices as the number
of possible edges and non-edges is $\binom{7}{2}=21$ which is an odd number.

Since $\overline{\s_1}=\s_2$, we add the constraints $Q_1(F_1,F_2) = Q_2(\overline{F_1},\overline{F_2})$ for every $F_1,F_2 \in \F^{\s_1}_5$. This makes $Q_2$ completely defined by $Q_1$.
Moreover, we add the constraints $Q_0(F_1,F_2) =  Q_0(\overline{F_1},{F_2}) = Q_0({F_1},\overline{F_2}) = Q_0(\overline{F_1},\overline{F_2})$ for every $F_1,F_2 \in \F^{\s_0}_4$.
This reduces the number of entries to round in the symmetric matrix 
$Q_0$ from ${21\choose 2}$ to ${11\choose 2}$. 

We reduced the number of constraints from $776$ to $388$, and we reduced the number of variables from ${21\choose 2}+2{72\choose 2}$ to ${11\choose 2}+{72\choose 2}$.
With these reductions, we managed to round the entries
in $Q_1,Q_2$ and $Q_3$ and thus we obtained a solution for \eqref{eq:27}.

The rounded matrices as well as programs computing all possible $X$ and performing the rounding process
can be obtained at \oururl. 

We give more details about the rounding step in the appendix.
\end{proof}

In~\eqref{eq:27}, we not only have that the minimum of $f(H)-c_H$ is $1/27$, which proves Theorem~\ref{thm:asymG}, 
but we also have the values of $f(H)-c_H$ for each $H$ in $\M_7$. 

Let $\L = \{H\in \M_7 : f(H)-c_H = 1/27\}$. We listed $\L$ in Figure~\ref{fig-tight}.
We have the following proposition for graphs not in $\L$.

\begin{prop}\label{prop:forbiddenGraph}
Let $G$ be an admissible graph of order $n\to\infty$ such that $f(G)=\frac1{27}+o(1)$. 
If $H\in\M_7\setminus\L$, then $ p(H,G) = o(1)$.
\end{prop}

\begin{proof}
Using \eqref{eq:lowerBound}, we have that 
\[
\frac1{27}+o(1) = f(G) +o(1) \ge \sum_{H\in \M_l} (f(H)-c_H)\cdot  p(H,G).
\]
In this section, we showed that by choosing $l=7$ and types $\s_0,\s_1,\s_2$ we have $\min_{H\in \M_7} (f(H)-c_H) = 1/27$. 
Then since $\sum_{H\in \M_l} p(H,G) = 1$, we know that if $f(H)-c_H>1/27$, then $p(H,G) = o(1)$.
\end{proof}

Notice that the Proposition~\ref{prop:forbiddenGraph} can be stated equivalently as follows.

\begin{prop}\label{prop:forbiddenGraphEps}
For every $\delta > 0$ there exists $n_0 = n_0(\delta)$ and $\eps' > 0$ such that 
for every admissible graph $G$ of order $n > n_0$ with $f(G) < 1/27 + \eps'$, 
if $H\in\M_7\setminus\L$, then $p(H,G) < \delta$.
Note that it is sufficent to pick
\[ \eps' < \delta \cdot \min_{H \in \M_7 \setminus \mathcal{L}}\{f(H)-c_H-1/27\}.\] 
\end{prop}

Proposition~\ref{prop:forbiddenGraphEps} will help us to get the
stability property of admissible graphs $G$ with $f(G)=\frac{1}{27}+o(1)$, which is discussed in the next section. 

\begin{figure}[htpd]
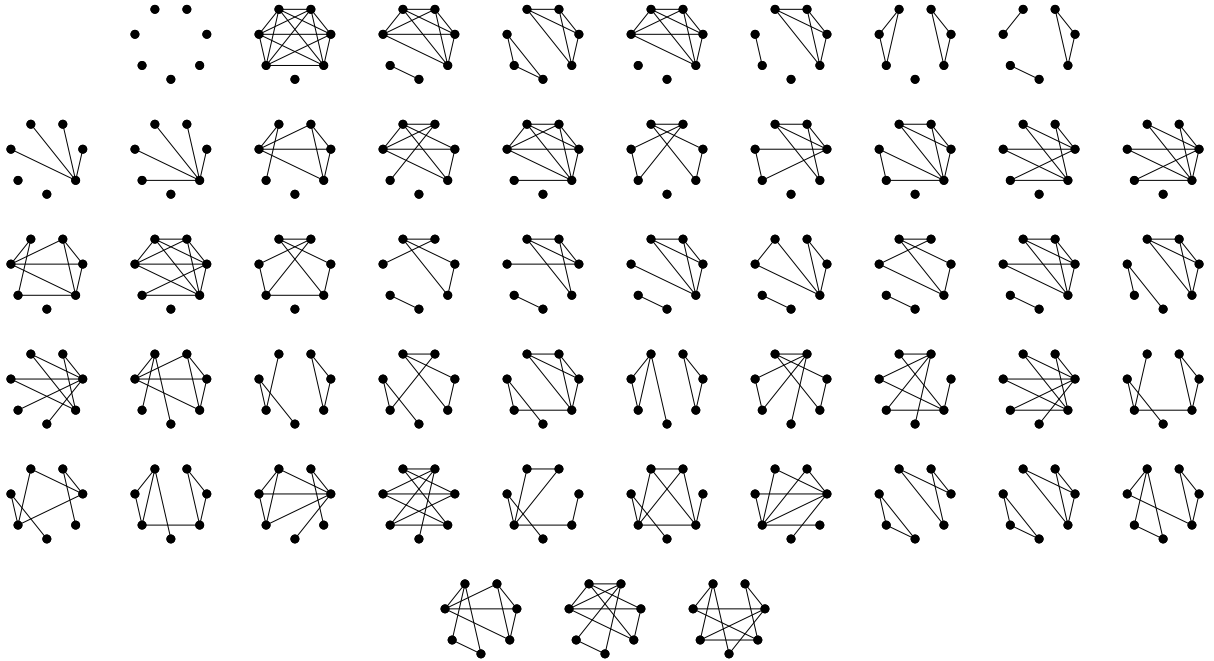

\begin{center}
\setstretch{3} 
\foreach \n in {1,...,8}{ \includegraphics[scale=0.7,page=\n]{fig-tight}\hskip 1em}\\
\foreach \n in {9,...,51}{ \includegraphics[scale=0.7,page=\n]{fig-tight}\hskip 1em}
\end{center}
\caption{Graphs in $\L$. The first eight graphs are induced subgraphs of $T_3(n)$ or $\overline{T}_3(n)$.
 In order to save space, a depicted graph $H$ represents both $H$ and $\overline{H}$.}\label{fig-tight}
\end{figure}

\section{Stability Property}\label{sec:stability}
In this section we will prove the following stability type statement.
\bth\label{thm:stability}
For every $\eps > 0$ there exist $n_0$ and $\eps' > 0$ such that every admissible graph $G$ 
of order $n > n_0$ with $f(G) \leq \frac1{27}+\eps'$, 
is isomorphic to either $T_3(n)$ or $\OO{T_3(n)}$ after adding and/or deleting at most $20\eps n^2$ edges.
\eth

We will use our flag algebra results from Section~\ref{sec:FA} and the infinite removal lemma to prove Theorem~\ref{thm:stability}. 
The infinite removal lemma, proved by Alon and Shapira~\cite{Alon:2008}, is a substantial generalization of the induced removal lemma.

\blm[Infinite Removal Lemma~\cite{Alon:2008}]\label{lem:InfRem}
For any (possibly infinite) family $\HH$ of graphs and $\eps >0$, there exists $\delta>0$ such that 
if a graph $G$ on $n$ vertices contains at most $\delta n^{v(H)}$ induced copies of $H$ for every graph $H$ in $\HH$, 
then it is possible to make $G$ induced $H$-free, for every $H\in\HH$, by adding and/or deleting at most $\eps n^2$ edges.
\elm

\bpf[Proof of Theorem~\ref{thm:stability}]
Fix an $\eps > 0$.
Let $\delta$ be from Lemma~\ref{lem:InfRem}, when applied with $\eps$ and $\HH = (\M_7 \setminus \L) \cup \{ \textrm{not admissible graphs}  \}$.
Let $\eps'< \eps^4$ and $n_1$ be given by Proposition~\ref{prop:forbiddenGraphEps} such that $p(H,G) < \delta$ for every $H\in\M_7\setminus\L$
and $G$ on at least $n_1$ vertices. Let $n_0 > n_1$ such that  $f(G)>1/27-\eps'$ for all $G$ of order at least $n_0/2$.
Notice that every non-admissible graph $H$ satisfies $p(H,G) = 0$  for every admissible $G$.

Let $G$ be an admissible graph of order $n > n_0$ with $f(G) \leq \frac1{27}+\eps'$.
Now we apply Lemma~\ref{lem:InfRem} and conclude that by adding and/or deleting at most $\eps n^2$ edges,
every induced subgraph of $G$ on $7$ vertices belongs to $\L$ and $G$ is still admissible.

By direct inspection of graphs in $\L$, we have the following two properties of $G$. 
Notice that if all $7$-vertex induced subgraphs of $G$ satisfy these two properties, then so does $G$. 
Also notice that a graph $H$ satisfies these two properties if and only if $\OO H$ satisfies them.
\begin{description}
\item[Property A:] There are no $K_4$ and $\OO K_4$ that share a vertex. 
\item[Property B:] For every pair of $K_4$'s that share at least one vertex, the union of their vertex sets spans a clique. 
For every pair of $\OO K_4$'s that share at least one vertex, the union of their vertex sets spans an independent set.
\end{description}

Let $(G,x)$ be the $1$-flag where vertex $x$ is the labeled vertex, then 
$P(K_4^1,(G,x))$ is the number of $K_4$'s in $G$ that contain $x$. Define 
\begin{align*}
 F(x,G)=P(K_4^1,(G,x))+P(\OO K_4^1,(G,x))\text{and} f(x,G)=\left.F(x,G)\middle/\binom{v(G)-1}{3}\right. .
\end{align*}
Then we have $f(G) = \left(\sum_x f(x,G)\right)/v(G)$.
Let $G_0 = G$. 
For $i \geq 0$, let $x_i$ be the vertex with largest $f(x_i,G_i)$. 
If $f(x_i,G_i) > 1/27+2\eps'/\eps$, we create $G_{i+1}$ from $G_i$
by removing vertex $x_i$. If $f(x_i,G_i) \leq 1/27+2\eps'/\eps$, we
define $G' = G_i$ and $d=i$.
Note that $f(G_{i}) \leq f(G_{i-1})$, so $f(G_i)\le f(G)\le 1/27+\eps'$.
Also notice, that the process is not deterministic if there are
more candidates for $x_i$ for some $i$ (any choice of $x_i$ will work).
\begin{claim}
$d < \eps n$.
\end{claim}
\bpf
Denote $v = v(G_{i-1})$ and $y$ the vertex deleted from $G_{i-1}$.
Then 
\begin{align}
f(G_{i-1}) - f(G_i) &\geq \frac{ 4\eps'}{\eps n},\label{eq:fGi}
\end{align}
which follows from the following computation
\begin{align*}
f(G_{i-1}) - f(G_i) &= \frac{\sum_x f(x,G_{i-1})}{v} - \frac{\sum_x f(x,G_{i})}{v-1} \\
                            &= \frac{f(y,G_{i-1}) + \sum_{x\neq y} f(x,G_{i-1})}{v} - \frac{\sum_x f(x,G_{i})}{v-1} \\
                            &= \frac{f(y,G_{i-1}) + (\sum_{x\neq y} F(x,G_{i-1}))/\binom{v-1}{3}}{v} - \frac{\sum_x f(x,G_{i})}{v-1} \\
                            &= \frac{4f(y,G_{i-1}) + \binom{v-2}{3}(\sum_{x} F(x,G_i))/(\binom{v-2}{3}\binom{v-1}{3})}{v} - \frac{\sum_x f(x,G_{i})}{v-1} \\
                            &= \frac{4f(y,G_{i-1}) + \frac{v-4}{v-1}\sum_{x} f(x,G_i)}{v} - \frac{\sum_x f(x,G_{i})}{v-1} \\
                            &= \frac{4(v-1)f(y,G_{i-1}) + (v-4)\sum_x f(x,G_{i}) -  v\sum_x f(x,G_{i})  }{v(v-1)}\\
                            &= \frac{4f(y,G_{i-1}) -  4f(G_i)}{v}\\
                            &\geq \frac{4\left( 2\eps'/\eps - \eps'\right)}{n} 
                            \geq \frac{ 4\eps'}{\eps n}.
\end{align*}

If follows from $n \geq n_0$ that  $f(H) > 1/27-\eps'$ for every admissible graph $H$ on at least $n/2$ vertices.
However, if $d > \eps n$, then for $i = \eps n$, $f(G_i) < 1/27+\eps' - 4i\eps'/\eps n < 1/27- \eps'$, which is a contradiction
since $G_i$ has at least $n-\eps n = (1-\eps)n \geq n/2$ vertices.
\epf

\begin{claim}
The number of vertices $x$ with $f(x,G') < \frac1{27} - \eps$ is at most $\eps v'$, where $v' = v(G')$.
\end{claim}
\bpf
Let the number of vertices with $f(x,G') < \frac1{27}-\eps$ be $z$. 
\begin{align*}
v'f(G')& = \sum_x f(x,G')  
       < z\left(\frac1{27} - \eps\right) + (v'-z)\left(\frac1{27} + \frac{2\eps'}{\eps}\right)\\
       &= -z\eps + v'\frac1{27} + v'\frac{2\eps'}{\eps} - z\frac{2\eps'}{\eps}        
       < \frac{v'}{27} + \frac{2v'\eps'}{\eps} - \eps z.   
\end{align*}
If $z > \eps v'$, then we get
\[
f(G') < \frac{1}{27} + \frac{2\eps'}{\eps} - \eps^2   < \frac{1}{27} - \eps',
\]
which is a contradiction (recall that $\eps' < \eps^4$).
\epf

Let $G''$ be the graph obtained from $G'$ by removing all such vertices. We removed at most $\eps v'$ vertices, so 
\[ F(x,G') - F(x,G'') < \left. \eps v'\binom{v'-2}{2}  \right. \] 
for each vertex $x\in V(G'')$. 
Denote $v(G'')$ by $v''$. We have $v''\ge (1-\eps)v'$ and 
\begin{align*}
F(x,G'')&> F(x,G') - \eps v' \binom{v'-2}{2} \\
& \ge \left( \frac{1}{27}-\eps\right)\binom{v'-1}{3} - \eps v'\binom{v'-2}{2}\\
& \ge \left( \frac{1}{27}-\eps\right)\frac{(v'-1)(v'-2)(v'-3)}{6} - \eps \frac{(v'-1)(v'-2)(v'-3)}{1.5}\\
& \ge \left( \frac{1}{27}-5\eps\right) \binom{v''-1}{3}.
\end{align*}

We know that $f(x,G')\le 1/27+2\eps'/\eps$. Then since $\eps'<\eps^4$
and $v''\ge (1-\eps)v'$, we have
\[
F(x,G'')\le F(x,G')\le \left( \frac{1}{27}+\frac{2\eps'}{\eps}\right) \binom{v'-1}{3} \le \left( \frac{1}{27}+5\eps\right) \binom{v''-1}{3}.
\]

This means that for every vertex $x\in V(G'')$, we have 
\begin{equation}\label{eq:reg}
 \left(\frac1{27}-5\eps\right)\binom{v''-1}{3}<F(x,G'')<\left(\frac1{27}+5\eps\right)\binom{v''-1}{3}.
\end{equation}

For $x,y\in V(G'')$, write $x\sim y$ if $x=y$ or there is a chain of vertex-intersecting $K_4$'s or $\OO K_4$'s connecting $x$ to $y$. 
Clearly, $\sim$ is an equivalence relation, and by Property A, each chain consists of cliques only or independent sets only. 
By Property B, each $\sim$-equivalence class is a clique or an independent set. 
Let $s$ be the size of the class of $x$. This means that $F(x,G'') = \binom{s-1}{3}$. It follows from~\eqref{eq:reg} that
\beq\label{eq:size}
\left(\frac{1}{3}-16\eps\right)v''<s<\left(\frac{1}{3}+16\eps\right)v'',
\eeq
which means each $\sim$-equivalence class has size at least $(1/3-16\eps)v''$ and at most $(1/3+16\eps)v''$.

Next, we claim that equivalence classes are all cliques or all independent sets. Suppose on the
contrary that $G''[A]$ is a clique and $G''[B]$ is an independent set. W.l.o.g., assume that the edge density between $A$
and $B$ is at least $1/2$. 
Then there exists a vertex $x$ in $B$ such that $|\Gamma(x)\cap A|\ge |A|/2$.
Taking a $4$-set $X\subset B$ containing $x$
and a $3$-set $Y\subset \Gamma(x)\cap A$,  
then $G''[X] = \OO K_4$ and $G''[Y\cup\{x\}] = K_4$. We find a $K_4$ and a $\OO K_4$ that share a vertex $x$,
contradicting Property A.

W.l.o.g., assume that each equivalence class is an independent set. It follows from~\eqref{eq:size} that there are exactly three equivalence classes. 
Denote them by $A_1,A_2$ and $A_3$. 
If there exist an $x \in A_i$ and $y_1,y_2,y_3 \in A_j$ ($i \neq j$) such that none of $xy_k$ is an edge, 
then $x \sim y_k$, which would contradict Property B. 
This means that all but at most $4v''$ of edges between equivalence classes are in $G''$. 
To get ${T_3(v'')}$, we need to add these edges and balance the three sets. In this step we change at most $4n+16\eps n^2$ edges.

We first change at most $\eps n^2$ edges of $G$ such that $G$ does not contain any $H\in \M_7\setminus \L$, then we remove at most $2\eps n$ vertices to form $G''$. Then we change at most $4n+16\eps n^2$ edges to get ${T_3(v'')}$. Therefore, to get ${T_3(n)}$ from $G$, we only need to change at most $\eps n^2 + 2\eps n^2 + 4n+16\eps n^2 \le 20\eps n^2$ edges, as required.
\epf

\section{Exact Result}\label{sec:exact}
We call $u\in V(G)$ a \emph{clone} of $v\in V(G)$ if $\Gamma(u) = \Gamma(v)$. In particular, $uv$ is not an edge of $G$.

\bpp\label{prop:clone}
Let $G$ be an admissible graph of order $n$. If we add a clone $x'$ of some $x\in V(G)$ to form a new graph $G'$ of order $n+1$,
i.e., $\Gamma_{G'}(x') = \Gamma_G(x)$, then $G'$ is still admissible. 
\epp
\bpf The graph $G$ comes from some permutation $\tau$ on $[n]$. Let $k$ be the number in $[n]$ that corresponds to $x$, 
then we can construct a new permutation $\tau'$ on $[n+1]$ as follows:
\begin{displaymath}
\tau'(i) = \left\{ \begin{array}{ll}
\tau(i) & \textrm{if $i\le k$ and $\tau(i)\le \tau(k)$}\\
\tau(i)+1 &\textrm{if $i\le k$ and $\tau(i)> \tau(k)$}\\
\tau(k)+1 & \textrm{if $i=k+1$}\\ 
\tau(i-1) & \textrm{if $i> k$ and $\tau(i-1)< \tau(k)$}\\
\tau(i-1)+1 & \textrm{if $i> k$ and $\tau(i-1)\ge \tau(k)$.}
\end{array} \right. \end{displaymath}
The {\pg} of $\tau'$ is $G'$ with $k+1$ corresponding to the new vertex $x'$.
\epf

Let $S$ be the 7-vertex graph obtained by gluing three paths $xy_iz_i$, $i=1,2,3$, at the common vertex $x$, see Figure~\ref{fig-S}.

\begin{figure}[htpd]
\begin{center}\includegraphics{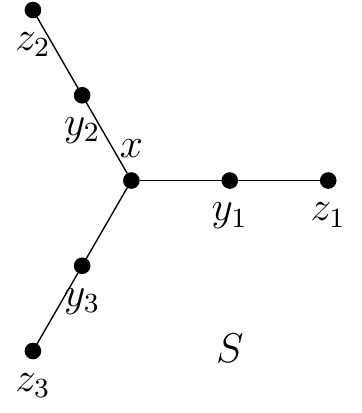} \end{center}
\caption{The graph S.}\label{fig-S}
\end{figure}

\bpp\label{prop:S} The graph $S$ is not admissible.
\epp
\bpf
Admissible graphs can be alternatively defined as intersection graphs of segments
whose endpoints lie on two parallel lines. For a vertex $v$ in $S$ denote by $s(v)$ the segment
representing $v$. Since $y_1,y_2$ and $y_3$ form an independent
set, segments representing them do not intersect. On the other hand $s(x)$ intersects all of them.
W.l.o.g., assume that $s(y_2)$ is middle of the three segments in the order they intersect $s(x)$,
see Figure~\ref{fig-segments}.
Since $z_2$ is adjacent only to $y_2$, the segment representing $z_2$ intersects only $s(y_2)$, which is clearly impossible.

\begin{figure}[htpd]
\begin{center}\includegraphics{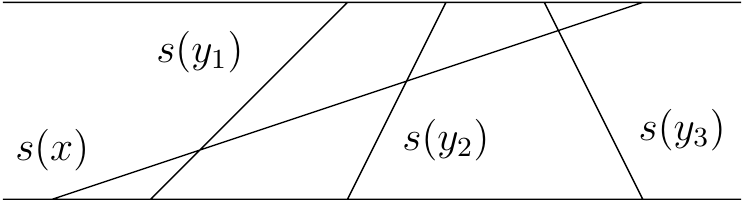} \end{center}
\caption{Representation of part of $S$ as intersecting segments.}\label{fig-segments}
\end{figure}

Alternatively, it is possible to check all admissible graphs on $7$ vertices, i.e., $\M_7$, and conclude that $S$ is not among them.
\epf

\begin{proof}[Proof of Theorem~\ref{thm:exactG}]
Let $G$ be an admissible graph of order $n$ with minimum $F$ among all admissible graphs on $n$ vertices, where $n$ is sufficiently large.
Fix $\eps > 0$  sufficiently small.
In particular, by Theorem~\ref{thm:asymG}, we assume $n$ large enough such that $f(G) \leq \frac{1}{27} +\eps$.
Let $V=V(G)$.

By Theorem~\ref{thm:stability}, we also assume that we can make $G$ equal to $T_3(n)$ by adding/or deleting at most 
$\eps n^2$ edges (also large $n$ needed).
Take a complete 3-partite graph $T$ on $V$ such that $|W|$ is minimized, where $W=E(T)\bigtriangleup E(G)$.
From Theorem~\ref{thm:stability} we know that 
$|W| < \eps n^2$ 
and $V$ can be partitioned into $V_1,V_2,V_3$ of sizes $(1/3-\eps)n < |V_i| < (1/3+\eps)n$,
which are the parts of $T$. Let $B=E(G)\setminus E(T)$
and $M=E(T)\setminus E(G)$.
We call edges in $B$ \emph{bad}, in $M$ \emph{missing}, and in $W$ \emph{wrong}.

\bpp\label{prop:32eps} For any $x\in V$ we have $f(x,G) \leq 1/27 + \frac{3}{2}\eps$.
\epp
\bpf 
First we prove that for any $x,y\in V$ we have 
\begin{align}
|F(x,G)-F(y,G)|\le \binom{n-2}{2}. \label{eq:FxFy}
\end{align}
W.l.o.g., assume $F(x,G) \geq F(y,G)$.
Let $G'$ be obtained from  $G$ by adding a clone of $y$ and removing $x$.
By Proposition~\ref{prop:clone}, $G'$ is an admissible graph.
By the extremality of $G$ we have
\[
0 \leq F(G')-F(G)\le F(y,G)-F(x,G)+\binom{n-2}{2},
\]
which gives \eqref{eq:FxFy}.

Recall, that $F(G)=\frac14\sum_{x\in V} F(x,G)\le F(T_3(n))$.
Suppose that there exists an $x$ such that $f(x,G) > 1/27 + \frac{3}{2}\eps$, i.e.
\begin{align*}
F(x,G) &> \left(\frac{1}{27} + \frac{3}{2}\eps \right)\binom{n-1}{3}.
\end{align*}
By using \eqref{eq:FxFy}  we have that for every $y\in V$
\begin{align*}
F(y,G) > \left(\frac{1}{27} + \frac{3}{2}\eps \right)\binom{n-1}{3} - \binom{n-2}{2} > \left(\frac{1}{27} + \eps \right)\binom{n-1}{3}.
\end{align*}
Hence $f(G) > \frac{1}{27} + \eps$ which contradicts our assumption that $f(G) \le \frac{1}{27}+\eps$.
\epf

\blm\label{lem:W} 
The graph $W$, and thus also $B$ and $M$, 
has maximum degree less than $\eta n$, where $\eta = 2\eps^{1/18}$. 
\elm
\bpf Let $x$ be a vertex in $V$. Let $\alpha_i=|\Gamma(x)\cap V_i|/|V_i|$ where $V_1,V_2,V_3$ are the parts of $T$. 
Let $\delta = 2\eps^{1/6}$.
If every $\alpha_i \in (\delta, 1-\delta)$
then there are at least $\delta^6 (\frac{1}{3}-\eps)^6n^6$ ways to choose a set 
$U = \{y_1,y_2,y_3,z_1,z_2,z_3\}$ with $y_i\in V_i\setminus \Gamma(x)-x$ and $z_i\in V_i\cap \Gamma(x)-x$. 
The number of such sets $U$ which contain a wrong edge is at most $\binom{n}{4}|W|  <  \eps n^6/24$.
Since $\delta^6 (\frac{1}{3}-\eps)^6n^6 > \eps n^6/24$, there exists
$U$ that does not contain any edge of $W$, which means 
$U\cup\{x\}$ induces the complement of the non-admissible graph $S$ in $G$, 
a contradiction to Proposition~\ref{prop:S}.

W.l.o.g., we may assume that $\alpha_1<\delta$ or $\alpha_1 > 1-\delta$.

If $\alpha_1<\delta$, then the number of copies of $\OO K_4$ via $x$ whose other three vertices are in $A_1$ is at least
\begin{align*}
\binom{(1-\delta)(1/3-\eps)n}{3} \geq \left(1-\delta\right)^3\left(\frac{1}{3}-\eps\right)^3\binom{n}{3}\geq \left(\frac{1}{27}-\eps-\delta\right)\binom{n}{3}.
\end{align*}

Thus $x$ has to be connected to almost every vertex in $A_2\cup A_3$. 
To be more precise, assume $\alpha_i \le 1- \eta$ for for $i=1$ or $i=2$.
Then we should have
\beq\label{eq:eta}
\binom{\eta(1/3-\eps)n}{3}\le\binom{(1-\alpha_i)|V_i|}{3} < \left( \frac{1}{27}+\eps\right)\binom{n}{3} - \left(\frac{1}{27}-\eps-\delta\right)\binom{n}{3}.
\eeq
However, \eqref{eq:eta} does not hold since $\eta = 2\eps^{1/18}$, so we know that $\alpha_2 > 1- \eta$ and $\alpha_3 > 1-\eta$.

If $\alpha_1 > 1 - \delta$, then 
\begin{align*}
f(x,G) &\ge\left. \left( \sum_{i=1}^3\binom{(1-\alpha_i)|V_i|}{3}  +\alpha_1\alpha_2\alpha_3|V_1||V_2||V_3|-\eps n^3\right)\middle/\binom{n}{3}\right.\\
&\ge\left. \left(  \binom{(1-\alpha_2)|V_2|}{3} + \binom{(1-\alpha_3)|V_3|}{3}  +(1-\delta)\alpha_2\alpha_3|V_1||V_2||V_3|-\eps n^3 \right)\middle/\binom{n}{3}\right.\\
&= (1/3-\eps)^3\left( (1-\alpha_2)^3 +  (1-\alpha_3)^3 +  6(1-\delta)\alpha_2\alpha_3  \right) - 6\eps\\
&\ge (1/3-\eps)^3\left( (1-\alpha_2)^3 +  (1-\alpha_3)^3 +  5\alpha_2\alpha_3  \right)-6\eps.
\end{align*}

Let $h(x,y) =  (1-x)^3 +  (1-y)^3 +  5xy$.
The minimum value of the polynomial $h(x,y)$ on $[0,1]^2$ is $1$ with equality if and only if $\{x,y\}=\{0,1\}$. 
We know that $f(x,G)\le 1/27+\frac{3}{2}\eps$.
Then by the continuity of $h$ and the compactness of $[0,1]^2$, $\{\alpha_2,\alpha_3\}$ is close to $\{0,1\}$.
W.l.o.g., assume $\alpha_2$ is close to $1$ and $\alpha_3$ is close to $0$. Let $\gamma = 6\eps^{1/3}$. If $\alpha_2\le 1-\gamma$ or $\alpha_3\ge\gamma$, then
\begin{align*}
f(x,G) &\ge  (1/3-\eps)^3\left( (1-\alpha_2)^3 +  (1-\alpha_3)^3 +  5\alpha_2\alpha_3  \right)-6\eps.\\
&= (1/3-\eps)^3 (1+(2 - 5(1-\alpha_2))\alpha_3 + 3\alpha_3^2-\alpha_3^3 + (1-\alpha_2)^3) - 6\eps \\
&\ge (1/3-\eps)^3 (1+\alpha_3 + (1-\alpha_2)^3)-6\eps\\
&> \frac{1}{27}+\frac{3}{2}\eps,
\end{align*}
which is a contradiction with Proposition~\ref{prop:32eps}. So $\alpha_2 > 1-\gamma$ and $\alpha_3< \gamma$.

Note that $\eta>\delta >\gamma$,
so now we know that two of $\alpha_1, \alpha_2, \alpha_3$ are at least $1-\eta$ and the other one is less than $\eta$.
W.l.o.g., we may assume $\alpha_1 < \eta, \alpha_2 > 1 -\eta$ and $\alpha_3 > 1 - \eta$. 
Then we know that $x\in V_1$ since otherwise we can move $x$ to $V_1$ and 
decrease $|W|$ which is a contradiction to the choice of $T$. Thus $d_W(x) < \eta n = 2\eps^{1/18} n$. 
\epf

It follows from Lemma~\ref{lem:W} that every bad edge $xy\in B$ belongs to at least $(1/9-\eta)n^2$ copies of $K_4$,
because
     \begin{align*}
     \left(\frac{1}{3}-\eps\right)^2n^2 - 2\eta n \left(\frac{1}{3}+\eps\right)n 
     = \left(\frac{1}{9} - \frac{2}{3}\eps + \eps^2 - \frac{2\eta}{3} - 2\eta\eps \right)n^2 
      \geq \left(\frac{1}{9} - \eta\right)n^2.
    \end{align*}

On the other hand, if we remove $xy$ from $E(G)$, this would create at most $(1/18+\eta^2)n^2$ copies of $\OO K_4$,
because
   \begin{align*}
   \binom{\left(1/3+\eps\right)n}{2} + \binom{\eta n}{2} 
   \leq \left(\frac{1}{18} + \frac{\eps}{3} +\frac{\eps^2}{2} + \frac{\eta^2}{2}\right)n^2 
   \leq \left(\frac{1}{18} + \eta^2 \right)n^2.
   \end{align*}
Also, by $\Delta(B)< \eta n$ and $b=|B| < \eta n^2$,
the number of $4$-sets that contain at least two bad edges is at most
   \begin{align*}
   \binom{b}{2} + 2b(\eta n)n \leq \frac{b^2}{2} + 2\eta bn^2 < \frac{\eta b n^2}{2} +  2\eta bn^2 < 3\eta bn^2.
   \end{align*}
Thus if $G'$ is obtained from $G$ by removing all bad edges of $G$, it satisfies $F(G')-F(G) \leq -bn^2/18+\eps b n^2<0$ unless $b=0$,
because 
   \begin{align*}
    F(G')-F(G) &\leq \left(\frac{1}{18}+\eta^2\right)bn^2  - \left(\left(\frac{1}{9}-\eta\right)bn^2 - 6 \cdot 3\eta bn^2\right) \\
    & \leq \left( -\frac{1}{18} + \eta^2 +19\eta\right)bn^2
     \leq -\frac{bn^2}{100}.
   \end{align*}
   
Clearly, the complete $3$-partite graph $T$ can be obtained from $G'$ by adding all missing edges between parts. Thus we have
$P(\OO K_4,T)\le P(\OO K_4,G')$. Then since $P(K_4,T)=0\le P(K_4,G')$, the admissible graph $T$
satisfies $F(T)\le F(G')<F(G)$ unless $b=0$. By the choice of $G$, we have $F(G)\le F(T)$, so $b=0$, which means $G$ is a subgraph of $T$ and $G$ is a $3$-partite graph. Then since $F(G)\le F(H)$ for every $H\in\M_n$, we know that $G$ is a subgraph of $T_3(n)$ and $F(G) = F(T_3(n))$.
\end{proof}
 
\section{Conclusion}
In~Theorem~\ref{thm:exact}, we verified Conjecture~\ref{con:exact} for $k=3$ and $n$
sufficiently large, and we fully characterized the set of the extremal
configurations. While revising our paper, we discovered that using a slightly
refined set-up of the flag algebra framework, we can also prove an analogue
of Theorem~\ref{thm:asym} for $k=4$ and $k=5$. We address these two cases
in a forthcoming note.

 \section*{Acknowledgement}
 We would like to thank Dan Kr\'a\soft{l} for fruitful discussions and developing the first version of the program
 implementing flag algebras over permutations. We would also like to thank Andrew Treglown for fruitful discussions.
 We also thank the anonymous referees for carefully reading the manuscript and for their valuable comments, which
 greatly improved the presentation of the results.

\bibliographystyle{abbrv}
\bibliography{refs.bib}

\clearpage
\begin{appendix}
\section{Rounding approximate solutions to exact solutions}
Recall from the proof of Theorem~\ref{thm:asymG} that a solution consists of several positive semidefinite matrices. For example, in our problem, our solution consists of three matrices $Q_0,Q_1$ and $Q_2$.
For simplicity, when describing the rounding procedure, we assume that there is only one matrix.
A computer solver can only solve a semidefinite program numerically  and thus we get an approximate solution. 
Let $Q$ be a $t$-by-$t$ matrix computed by a computer solver.
To make the solution exact, we need to convert entries in the matrix to rational numbers.
A resulting \emph{rounded} matrix $Q'$ must satisfy the following.
\begin{description}
\item[Goal $1$.] $Q'$ is positive semidefinite.
\item[Goal $2$.] $Q'$ gives us the desired number, i.e., see \eqref{eq:27}.
\end{description}

The idea of the rounding is following. For most of entries in $Q'$ we use a rational number close to the corresponding entry in $Q$.
The other entries in $Q'$ will be computed such that $Q'$ satisfies Goals $1$ and $2$.

We will construct a system of linear equations whose variables are entries of $Q'$ (ignoring the entries below the main diagonals) 
and all constants are rational numbers. There are two types of linear equations in the system, Type $1$ and Type $2$, 
which make our solution achieve Goal $1$ and Goal $2$ respectively. 
We again use computer solver to solve the linear system, but unlike a semidefinite program, 
a system of linear equations can be solved over rationals.

When we use an entry from $Q$, it is sufficient in our case if the corresponding entry in $Q'$ differs by at most $\eps = 10^{-5}$.

To achieve Goal $1$, we want all eigenvalues to be non-negative. 
We know that $Q$ is positive semidefinite, so all its eigenvalues are non-negative. 
If an eigenvalue of $Q$ is a large positive number compared to $\eps$, 
then  we expect it to be still positive after rounding, since as we mentioned above, entries of $Q$ are perturbed just a little bit.
But if an eigenvalue of $Q$ is small, for example, $10^{-6}$, 
then it may become $-10^{-8}$ after rounding and $Q'$ would not be positive semidefinite.
To avoid this, we force such eigenvalues to become $0$ after rounding. 
We do this by adding a constraint to our linear system for every such eigenvalue. 
Let $\{X_i\}$ be the set of eigenvectors of $Q$ whose eigenvalues are smaller than $\eps_1$ for some $\eps_1 > 0$.
We assume that $X_i$ is close to an eigenvector of $Q'$ with eigenvalue $0$.
So we find an approximate basis $\{X'_i\}$ of the linear space generated by $\{X_i\}$, and add $Q'X'_i = 0$ to our linear system.
These are Type $1$ linear equations.
Note that entries of $Q'$ are variables, so this  gives us $t$ linear equations for each $X'_i$. 
Let $X_i = [x_{i,1},\ldots,x_{i,t}]$. The algorithm of finding $X'_i$ is outlined below, which is taken from Baber's Thesis~\cite{Baber}:
\begin{tabbing}
Fo\=r each $X_i$:\\
\> Let $\ell$ be  $\operatorname{arg\,max}_j |x_{i,j}|$.\\
\> Set $X_i = X_i/x_{i,\ell}$.\\
\> Fo\=r all $k\ne i:$\\
\>\>Set $X_k = X_k - x_{k,\ell}\cdot X_i$.\\
Guess $X'_i$ from $X_i$ by assuming that all entries of $X'_i$ are rational numbers.
\end{tabbing}

More details of the algorithm are in Section 2.4.2.2 of~\cite{Baber}. The last step of the algorithm means that $X_i$ should look good and one can see instantly from $X_i$ what the exact value is. For example, if one sees $0.33333332$, then $1/3$ should be guessed.

To achieve Goal $2$, we check values of $f(H)-c_H$ for all $H$ in $\M_l$.
If $f(H)-c_H$ is much larger than $1/27$, we hope that it will be still larger than $1/27$ after rounding.
However, if $f(H)-c_H$ is close to $1/27$, a small change on entries of $Q$ could result in $f(H)-c_H$ being less than $1/27$, which violates Goal $2$.
To prevent this, we add a linear equation $f(H)-c_H = 1/27$ for every $H \in \M_l$ if $f(H)-c_H < 1/27 + \eps_2$ for some $\eps_2 > 0$.
These are Type $2$ linear equations.

The system of $k$ Type 1 and Type 2  linear equations can be written as
$$ 
A y = b,
$$
where $y=\{y_1,\ldots,y_m\}$ corresponds to entries of $Q'$, $A \in \mathbb{Q}^{k\times m}$, and $b \in \mathbb{Q}^k$.
Usually, $m$ is larger than $r = rank(A)$. 
W.l.o.g., assume that the first $r$ columns of $A$ are linearly independent.
Then  $A$ can be written as  $\begin{bmatrix} A' & A''\end{bmatrix}$ where $A'$ is the first $r$ columns of $A$ and $A''$ is the rest of the columns of $A$. 
Let $y' = \{y_1,\ldots,y_{r}\}$ and $y'' = \{y_{r+1},\ldots,y_m\}$.  
We assign to $y_i$ in $y''$ a rational number, such that $|y_i - x_i| < \eps_3$, where $x_i \in Q$ corresponds to $y_i$ and $\eps_3 > 0$.
This step can be done arbitrarily. For example, let $\eps_3 = 10^{-5}$ and keep the first $5$ digits of $x_i$. 
Then we have the following matrix equation:
\begin{equation}\label{eq:round}
A' y' = b - A''y''.
\end{equation}
Note that the number of equations in~\eqref{eq:round} may be larger than $r$. So this system may have no solution. But if it has
a solution, then this solution is unique,
which gives a matrix over rational numbers.
Then we need to verify if this matrix satisfies Goals $1$ and $2$. If yes, we get $Q'$.
If not, we can try to redo the computation with a smaller $\eps_3$, or look which of the goals is violated and enlarge $\eps_1$ or $\eps_2$ to add more equations to the linear system.

If we are unlucky that the linear system~\eqref{eq:round} has no solution, then it means we added too many equations. Note that we pick eigenvalues that are smaller than $\eps_1$ and add corresponding Type $1$ equations, and pick $H$ with $f(H)-c_H<1/27+\eps_2$ and add corresponding Type $2$ equations. 
In order to have fewer equations,
we may re-pick Type $1$ and Type $2$ equations with smaller $\eps_1$ and $\eps_2$.

So far in our rounding procedure, we get Type $1$ and Type $2$ equations only from $Q$.
If the attacked problem has conjectured extremal structures,
we can also get Type $1$ and Type $2$ equations from those structures.

Take our problem for example.
Let $G$ be an extremal graph on $n$ vertices. By Proposition~\ref{prop:forbiddenGraph}, if $p(H,G)>o(1)$, then $f(H)-c_H = 1/27$, which gives Type $2$ equations. For our problem, this gives the first eight graphs in Figure~\ref{fig-tight}.
Unsurprisingly, every $H$ of these eight graphs satisfy $f(H)-c_H\approx 1/27$ in  $Q$. 
So these Type $2$ equations are usually generated from $Q$ by the process described before.

For Type $1$ equations, using~\eqref{eq:lowerBound}, we have
\begin{align*}
1/27+o(1) = f(G) + o(1) &\ge \sum_{H\in\M_7} (f(H)-c_H)\cdot p(H,G)
\end{align*}
and
\begin{align*}
f(G) - \sum_{H\in\M_7}c_Hp(H,G)  = \sum_{H\in\M_7} (f(H)-c_H)\cdot p(H,G)  \ge \min_{H\in\M_7} (f(H)-c_H) = 1/27.
\end{align*}
This gives $\sum_{H\in\M_7}c_Hp(H,G) = o(1)$.
Recall from Section~\ref{sec:FA} that 
we use $(\s_i,X_i, Q_i)$ to get $c_H$. Denote $X_i = \{F^i_j\}$. For $\theta\in\Theta(|\s_i|,G)$, let $Y_{i,\theta}$ be the vector whose entries are $p(F^i_j,(G,\theta))$. It follows from \eqref{eq:Q} and the definition of $c_H$ that

\beq\label{eq:egvector}
o(1) = \sum_{H\in\M_7}c_Hp(H,G) = \sum_i E_{\theta\in \Theta(|\s_i|,G)} Y_{i,\theta}^T\cdot Q_i \cdot Y_{i,\theta}.
\eeq

For each $\theta\in\Theta(|\s_i|,G)$ we have a vector $Y_{i,\theta}$, 
but if the conjectured extremal structures are symmetric in some sense, then there may be only $C$ different $Y_{i,\theta}$'s where $C$ is a constant independent of $n$.
Choose $\theta$ from $\Theta(|\s_i|,G)$ uniformly at random. If $\P[Y_{i,\theta} = Y_{i,\phi}]>o(1)$ for some $\phi \in\Theta(|\s_i|,G)$, then we have $Y_{i,\phi}^T\cdot Q_i\cdot Y_{i,\phi} = 0$, otherwise we do not have~\eqref{eq:egvector}. 
Since $Q_i\succeq0$, this means that $Y_{i,\phi}$ is an eigenvector of $Q_i$ with eigenvalue $0$,
giving us Type $1$ equations. 
In our problem, the vectors $\{Y_{i,\phi}\}$ we get from conjectured extremal structures are in the space generated by $\{X'_i\}$. So there is no need to combine equations generated from these two methods.
Let us mention that, for our problem, we could not guarantee that the rounded matrix is positive definite
by just using Type $1$ equations that come from the Tur\'an graph. We also
needed Type $1$ equations from the numerical solution.

\end{appendix}

\end{document}